\documentclass{article}
\usepackage{latexsym}
\usepackage{epsfig}
\usepackage{amsmath}
\usepackage{amssymb}
\usepackage{amsthm}
\usepackage{caption}
\newtheorem{theorem}{Theorem}[section]

\newtheorem{lemma}[theorem]{Lemma}

\theoremstyle{remark}

\theoremstyle{assumption}
\newtheorem{assumption}[theorem]{Assumption}

\theoremstyle{definition}

\begin{document}

\title{Sufficient conditions for the existence of exponential-polynomial expansions for solutions of certain differential equations}

\author{Roland Hildebrand  (khildebrand.r@mipt.ru) \\ Rahaf Habib (rahoufhabib@gmail.com) \\ MIPT, FPMI, 141701 Dolgoprudny}

\maketitle

\begin{abstract}
We consider ordinary differential equations (ODE) of the form $u''u - (u')^2 = e^{-x}P(u) - 1$, where $P$ is a polynomial. In previous work, necessary conditions on $P$ have been established for certain families of solutions of these ODEs to have asymptotic expansions of the form $u(x) = \sum_{k=0}^{\infty} p_k(x+c)e^{-kx}$ for $Re\,x \to +\infty$, where $c \in \mathbb C$ is an arbitrary constant parameterizing the solution family, and $p_k$ are polynomials, with $p_0(x) = x$. These conditions amount to $P(0) = 0$ and $P'(0) = \frac12P''(0)$. Here we show that these two conditions are also sufficient. The results imply the existence of corresponding expansions for certain degenerate Painlev\'e III transcendents. 
\end{abstract}



\section{Introduction}

In previous work \cite[Section 2]{NecessaryConditions} we have shown that certain degenerate Painlev\'e III equations are equivalent to members of a class of ordinary differential equations (ODEs). The ODEs in the class have the form
\begin{equation} \label{generalODE}
    u''u - (u')^2 = e^{-x}P(u) - 1,
\end{equation}
where $P$ is a polynomial and serves as parameter of the class. We are interested in solutions of ODE \eqref{generalODE} with asymptotics
\begin{equation} \label{roughAsymptotics}
    u(x) \sim x + c, \qquad x \to +\infty,
\end{equation}
where $c \in \mathbb C$ is an arbitrary constant. Let us denote the solution of ODE \eqref{generalODE} with asymptotics \eqref{roughAsymptotics} by $u(x;c,P)$.

Let us remark that for $P = 0$ equation \eqref{generalODE} can be solved explicitly, and the generic solution is given by
\[ u(x) = \frac{e^{k(x+c)} - e^{-k(x+c)}}{2k},
\]
where $k \in \mathbb C \setminus \{0\}$, $c \in \mathbb C$ are parameters. In the limit $k \to 0$ we obtain the special solutions
\begin{equation} \label{P0solution}
u(x;c,0) = x + c.
\end{equation}

In this work we do not focus on the existence of the solutions $u(x;c,P)$, but rather on the existence of formal asymptotic series which fulfill the relation \eqref{generalODE} and have leading term \eqref{roughAsymptotics}. For certain polynomials $P$ equation \eqref{generalODE} is equivalent to a degenerate Painlev\'e III equation, and in these cases the existence of the solutions $u(x;c,P)$ is guaranteed by the properties of the corresponding Painlev\'e transcendents.

Previously we have derived necessary conditions on the polynomial $P$ in order for a formal asymptotic series of the form
\begin{equation} \label{fullAsymptotics}
    u(x;c,P) \sim \sum_{k=0}^{\infty} p_k(x+c;P)e^{-kx}
\end{equation}
to fulfill \eqref{generalODE}, where $p_k(x;P)$ are polynomials in $x$ depending on $P$ but not on $c$, and $p_0(x;P) = x$. Namely, we have the following result \cite[Theorem 5.2]{NecessaryConditions}.

\begin{theorem} \label{thm:mainNecessary}
    Let $c \in \mathbb C$ be arbitrary and let $u(x)$ be a solution of ODE \eqref{generalODE} admitting an asymptotic expansion of the form \eqref{fullAsymptotics}, where $p_0(x) = x$, and $P(t) = \sum_{j=0}^{d} \pi_jt^j$ is a polynomial with $\pi_d \not= 0$. Then $\pi_0 = 0$ and $\pi_1 = \pi_2$. In particular, $d \geq 2$. In this case the degrees of the polynomials $p_k$ are given by $\deg\,p_k = d_k = (d-2)k + 1$.
\end{theorem}



In this work we show that the conditions $\pi_0 = 0$ and $\pi_1 = \pi_2$ on the coefficients $\pi_j$ are also sufficient for the existence of a formal asymptotic series of the form \eqref{fullAsymptotics} (Theorem \ref{thm:mainSufficiency}). This implies that for all polynomials $P$ in a linear subspace of co-dimension 2 of the space of all polynomials, the solutions of ODE \eqref{generalODE} with asymptotics \eqref{roughAsymptotics}, if they exist, have an asymptotic expansion of the form \eqref{fullAsymptotics}. In particular, this holds for those ODEs \eqref{generalODE} which are equivalent to degenerate Painlev\'e III equations. This implies the existence of corresponding expansions for the concerned Painlev\'e transcendents.

The remainder of the paper is structured as follows. In Section \ref{sec:main} we prove the main result Theorem \ref{thm:mainSufficiency} on the sufficiency of conditions on $P$ for expansions of the form \eqref{fullAsymptotics} to exist. In Section \ref{sec:zeroBehaviour} we provide an interpretation of these conditions by studying the expansion of solutions $u(x)$ of ODE \eqref{generalODE} around their zeros. In Section \ref{sec:Painleve} we provide explicit expansions for the Painlev\'e transcendents which can be obtained by transforming solutions of \eqref{generalODE}. In Section \ref{sec:outlook} we give an outlook on further research directions.

\section{Solvability conditions} \label{sec:main}

The polynomials $p_k$ in expansion \eqref{fullAsymptotics} have to be computed one by one, by inserting the expansion into \eqref{generalODE} and matching the terms at $e^{-kx}$. At each step the linear system on the coefficients of $p_k$ is overdetermined, and the polynomial $P(t) = \sum_{k=0}^d \pi_kt^k$ has to fulfill certain constraints in order for this system to have a solution.

In \cite[Section 5]{NecessaryConditions} we have shown that the conditions $\pi_0 = 0$, $\pi_1 = \pi_2$ on the coefficients of $P$ are necessary and sufficient for polynomial solutions $p_1,p_2,p_3$ of the corresponding systems to exist. Here we  shall show by induction over $k$ that then also all other $p_k$ can be found. Let us make the following assumption, corresponding to the induction hypothesis.

\begin{assumption} \label{InductionHypothesis}
    Let $\deg\,P = d \geq 2$, $\pi_0 = 0$, $\pi_1 = \pi_2$, and $s \geq 1$. Assume that polynomials $p_k$ for $k \leq s-1$ have been determined such that insertion of the partial sum
    \[ U_{s-1}(x) = \sum_{k=0}^{s-1} p_k(x)e^{-kx}
    \]
    into \eqref{generalODE} in place of $u(x)$ yields a residual
    \begin{equation} \label{Dsx}
        D_s(x) := e^{-x}P(U_{s-1}) - 1 - U_{s-1}''U_{s-1} + (U_{s-1}')^2 = \sum_{k=s}^{d \cdot (s-1) + 1} q_{k,s}(x)e^{-kx}
    \end{equation} 
    with leading exponential term not larger than $e^{-sx}$. Here $q_{k,s}$ are also polynomials.
\end{assumption}

The maximal exponent $e^{-(d(s-1)+1)x}$ comes from the term $e^{-x}P(U_{s-1})$ in $D_s$, since the other terms give rise to exponents at most $e^{-2(s-1)}$.

\begin{lemma} \label{lem:summandCondition}
    The linear system on the next polynomial $p_s$ has a solution if and only if $(q_{s,s}e^{-sx})''|_{x = 0} = 0$.
\end{lemma}

\begin{proof}
    Let us deduce the linear system on the next polynomial $p_s$ in expansion \eqref{fullAsymptotics}. Note that $U_{s-1}$ and hence also $D_s$ are entire functions. Set 
\[ U_s(x) = U_{s-1}(x) + p_s(x)e^{-sx}
\]
and 
\begin{align*}
    D_{s+1} =\,& e^{-x}P(U_s) - 1 - U_s''U_s + (U_s')^2 
  = e^{-x}P(U_{s-1} + p_se^{-sx}) - 1 - \\ &- (U_{s-1}'' + (p_s'' - 2sp_s' + s^2p_s)e^{-sx})(U_{s-1} + p_se^{-sx}) + (U_{s-1}' + (p_s' - sp_s)e^{-sx})^2 \\
  =\,& D_s(x) + e^{-x}\left(P(U_{s-1} + p_se^{-sx}) - P(U_{s-1})\right) - U_{s-1}''p_se^{-sx} - (p_s'' - 2sp_s' + s^2p_s)e^{-sx}U_{s-1} \\ &+ 2U_{s-1}'(p_s' - sp_s)e^{-sx} + O(e^{-(s+1)x}).
\end{align*} 
Neglecting terms of order $e^{-(s+1)x}$ and higher, we can replace $e^{-x}P(U_{s-1} + p_se^{-sx})$ by $e^{-x}P(U_{s-1})$, $e^{-sx}U_{s-1}$ by $e^{-sx}x$, $e^{-sx}U'_{s-1}$ by $e^{-sx}$, $e^{-sx}U''_{s-1}$ by 0, and $D_s$ by $q_{s,s}e^{-sx}$. Then we obtain
\[ D_{s+1}(x) = -\left( (p_s'' - 2sp_s' + s^2p_s)x - 2(p_s' - sp_s) \right)e^{-sx} + q_{s,s}e^{-sx} + O(e^{-(s+1)x}).
\]
The polynomial $p_s$ has to be chosen such that the term at $e^{-sx}$ in $D_{s+1}$ vanishes. This leads to the linear system
\[ (p_s'' - 2sp_s' + s^2p_s)x - 2(p_s' - sp_s) = q_{s,s}
\]
on the coefficients of $p_s$. 

The linear operator of this overdetermined system has been studied in \cite[Lemma 5.1]{NecessaryConditions}. The system is solvable with respect to $p_s$ if and only if the polynomial $q_{s,s}(x) = \sum_{j=0}^{\deg\,q_{s,s}} c_{j;s,s}x^j$ fulfills a certain linear condition on its coefficients. Namely, the last three coefficients of $q_{s,s}$ have to fulfill the relation 
\begin{equation*}
    2c_{2;s,s}-2sc_{1;s,s}+s^2c_{0;s,s} = 0. 
\end{equation*}
This condition is easily seen to be equivalent to the condition in the lemma.
\end{proof}

The result can be stated as follows. The first summand, corresponding to $k = s$, in the second derivative $D_s''$ has to vanish at $x = 0$. However, in the neighbourhood of $x = 0$ the exponents $e^{-kx}$ are all of the same order $O(1)$, and hence indistinguishable by their order of smallness. The first summand therefore does not take on the role of a leading term. To circumvent this, we borrow a technique from \cite{Kitaev25} and introduce a second independent variable.

Introduce a parameter $a \in \mathbb C$ and consider the family of partial sums
\[ U_{s-1}(x;a) = \sum_{k=0}^{s-1} a^kp_k(x)e^{-kx},
\]
as well as the corresponding residual
\begin{equation} \label{Dsxa}
    D_s(x;a) := e^{-x}aP(U_{s-1}(x;a)) - 1 - U_{s-1}''(x;a)U_{s-1}(x;a) + (U_{s-1}'(x;a))^2.
\end{equation} 
Here the prime denotes derivation with respect to $x$.

\begin{lemma} \label{lem:DsCondition}
    The condition in Lemma \ref{lem:summandCondition} is fulfilled if and only if 
    \[ D_s''(x;a) = O(a^{s+1}) + O(xa^s)
    \]
    in the neighbourhood of $(x,a) = (0,0)$. 
\end{lemma}

\begin{proof}
By definition $U_{s-1}(x;a)$, and hence also $D_s(x;a)$, are entire functions in the variables $(x,a) \in \mathbb C^2$. In particular, $D_s''(x;a)$ can be developed in a globally converging Taylor series around $(x,a) = (0,0)$.

We have
\begin{equation} \label{Dsxaq}
    D_s(x;a) = \sum_{k=s}^{d \cdot (s-1) + 1} a^kq_{k,s}(x)e^{-kx}.
\end{equation} 
Indeed, expression $U_{s-1}(x;a) = \sum_{k=0}^{s-1} p_k(x)(ae^{-x})^k$ differs from expression $U_{s-1}(x) = \sum_{k=0}^{s-1} p_k(x)e^{-kx}$ by replacement of $e^{-x}$ by $ae^{-x}$. The same relation holds when comparing the derivatives of $U_{s-1}(x)$ and $U_{s-1}(x;a)$. In the same way expressions \eqref{Dsxa} and \eqref{Dsx} differ from each other. Equality \eqref{Dsxaq} then follows from the definition of the polynomials $q_{k,s}$ in \eqref{Dsx}. This implies
\begin{equation} \label{Dsxaq2}
    D_s''(x;a) = \sum_{k=s}^{d \cdot (s-1) + 1} a^k(q_{k,s}(x)e^{-kx})'' =: \sum_{k=s}^{d \cdot (s-1) + 1} a^k \cdot \Phi_{k,s}(x),
\end{equation} 
where $\Phi_{k,s}(x) = (q_{k,s}(x)e^{-kx})''$ are entire functions.

The condition in Lemma \ref{lem:summandCondition} then states that $\Phi_{s,s}(0) = 0$. Equivalently, the coefficient at $x^0a^s$ in the Taylor series of $D_s(x;a)$ vanishes. The proves the claim.
\end{proof}

\begin{lemma} \label{Phiss0}
    Under Assumption \ref{InductionHypothesis} we have $\Phi_{s,s}(0) = 0$.
\end{lemma}

\begin{proof}
    Since $U_{s-1}(x;0) = x$, hence $U'_{s-1}(0;0) = 1 \not= 0$, and $U_{s-1}(x;a)$ is analytic in $(x,a)$, by virtue of the implicit function theorem for sufficiently small $|a|$ there exists an analytic function $x_0(a)$ such that $U_{s-1}(x_0(a);a) \equiv 0$, with $x_0(0) = 0$. 

Differentiating \eqref{Dsxa}, we obtain
\begin{align*} 
    D_s'(x;a) &= -e^{-x}aP(U_{s-1}) + e^{-x}aP'(U_{s-1})U_{s-1}' - U_{s-1}'''U_{s-1} + U_{s-1}'U_{s-1}'', \\
    D_s''(x;a) &= ae^{-x}P(U_{s-1}) - 2ae^{-x}P'(U_{s-1})U_{s-1}' + ae^{-x}P''(U_{s-1})(U_{s-1}')^2 + ae^{-x}P'(U_{s-1})U_{s-1}'' \\ &- U_{s-1}^{(IV)}U_{s-1} + (U_{s-1}'')^2.
\end{align*}

Evaluating at $x = x_0(a)$ and using $U_{s-1}(x_0(a);a) = 0$, $P(0) = \pi_0 = 0$, $P'(0) = \pi_1$, $P''(0) = 2\pi_2 = 2\pi_1$, we get
\begin{align*}
    D_s(x_0(a);a) &= (U_{s-1}')^2 - 1, \\
    D_s'(x_0(a);a) &= e^{-x_0(a)}a\pi_1U_{s-1}' + U_{s-1}'U_{s-1}'', \\
    D_s''(x_0(a);a) &= - 2ae^{-x_0(a)}\pi_1U_{s-1}' + ae^{-x_0(a)}2\pi_1(U_{s-1}')^2 + ae^{-x_0(a)}\pi_1U_{s-1}'' + (U_{s-1}'')^2.
\end{align*}
From the first relation we get
\[ U_{s-1}'(x_0;a) = \sqrt{1 + D_s(x_0;a)},
\]
and then from the second one
\[ U_{s-1}''(x_0;a) = \frac{D_s'(x_0;a)}{\sqrt{1 + D_s(x_0;a)}} - e^{-x_0}a\pi_1.
\]
Inserting this into the third relation, we obtain
\begin{align*}
    D_s''(x_0;a) &= 2ae^{-x_0}\pi_1\left( 1 + D_s(x_0;a) - \sqrt{1 + D_s(x_0;a)} \right) + ae^{-x_0}\pi_1\left( \frac{D_s'(x_0;a)}{\sqrt{1 + D_s(x_0;a)}} - ae^{-x_0}\pi_1 \right) \\ &+ \left( \frac{D_s'(x_0;a)}{\sqrt{1 + D_s(x_0;a)}} - ae^{-x_0}\pi_1 \right)^2 \\
    &= 2ae^{-x_0}\pi_1\left( 1 + D_s(x_0;a) - \sqrt{1 + D_s(x_0;a)} \right) - ae^{-x_0}\pi_1\frac{D_s'(x_0;a)}{\sqrt{1 + D_s(x_0;a)}} + \frac{D_s'(x_0;a)^2}{1 + D_s(x_0;a)}.
\end{align*}
By \eqref{Dsxaq} we have $D_s = O(a^s)$, $D_s' = O(a^s)$. Inserting into above expression, we obtain $D_s''(x_0(a);a) = O(a^{s+1})$. 

On the other hand, since $x_0(a)$ is analytic, we have $\Phi_{k,s}(x_0(a)) = \Phi_{k,s}(0) + O(a)$. Relation \eqref{Dsxaq2} then gives
\[ D_s''(x_0(a);a) = \sum_{k=s}^{d \cdot (s-1)+1} a^k \Phi_{k,s}(x_0(a)) = \sum_{k=s}^{d \cdot (s-1)+1} a^k (\Phi_{k,s}(0) + O(a)) = a^s\Phi_{s,s}(0) + O(a^{s+1}).
\]

Combining with the preceding relation, we obtain $\Phi_{s,s}(0) = 0$, as claimed.
\end{proof}

We are now ready to prove the main result.

\begin{theorem} \label{thm:mainSufficiency}
    Let $\deg\,P = d \geq 2$, $\pi_0 = 0$, $\pi_1 = \pi_2$. Then for all $k \geq 1$ there exist polynomials $p_k(x;P)$ with degree $\deg\,p_k = (d-2)k+1$, such that with $p_0(x) = x$ and for every $c \in \mathbb C$ the formal series \eqref{fullAsymptotics} is a solution of ODE \eqref{generalODE}. Concretely, upon insertion of the formal series \eqref{fullAsymptotics} into \eqref{generalODE} the polynomials on the left- and right-hand sides at every exponent $e^{-kx}$ coincide.
\end{theorem}

\begin{proof}
Assume the conditions of the theorem. Let us show the base of the induction over $s$. For $s = 1$ we have
\[ U_{s-1} = x,\quad D_s = e^{-x}P(x).
\]
Hence Assumption \ref{InductionHypothesis} holds for $s = 1$.

The existence of a polynomial solution $p_s$ for general $s$ then follows by induction over $s$ from Lemmas \ref{lem:summandCondition}, \ref{lem:DsCondition}, and \ref{Phiss0}.
\end{proof}

\section{Behaviour of the solutions near a zero} \label{sec:zeroBehaviour}

Let us comment on the nature of the conditions on the coefficients $\pi_j$ of $P$. Since the $\pi_j$ are the Taylor coefficients of $P$ around $t = 0$, and in \eqref{generalODE} the polynomial $P$ is evaluated at $t = u(x)$, the constraints in Theorem \ref{thm:mainNecessary} must somehow be related to the behaviour of the solutions $u(x)$ near their zeros. This may seem peculiar, since the assertion of the theorem concerns solutions with asymptotic behaviour \eqref{roughAsymptotics}, implying $u \to +\infty$.
 
A link between the asymptotic behaviour and the zeros of $u(x)$ is provided by a symmetry (see \cite[Section 3.1]{NecessaryConditions}), which to every solution $u(x)$ of ODE \eqref{generalODE} assigns the solution $\tilde u(x;\delta) = u(x - \delta)$ of the ODE
\begin{equation*} 
    \tilde u''\tilde u - (\tilde u')^2 = e^{-x}e^{\delta}P(\tilde u) - 1,
\end{equation*}
where $\delta \in \mathbb C$ is an arbitrary complex constant. If $u(x)$ has asymptotics \eqref{roughAsymptotics}, then $\tilde u$ has asymptotics $x + c - \delta$. Therefore we have the identity
\[ u(x-\delta;c,P) = u(x;c-\delta,e^{\delta} \cdot P).
\]
In particular, we get
\begin{equation} \label{symmetrySolutions}
    u(x-c;c,P) = u(x;0,e^c \cdot P).
\end{equation} 
Inserting \eqref{fullAsymptotics}, we obtain the identity
\begin{equation*} 
    \sum_{k=0}^{\infty} e^{kc} \cdot p_k(x;P)e^{-kx} = \sum_{k=0}^{\infty} p_k(x;e^c \cdot P)e^{-kx},
\end{equation*}
which yields
\begin{equation} \label{symmetryPolynomials}
    p_k(x;a \cdot P) = a^k \cdot p_k(x;P), \qquad a = e^c.
\end{equation} 
For $a = e^c \to 0$ the right-hand side of \eqref{symmetrySolutions} becomes simply $u(x) = x$ by virtue of \eqref{P0solution}. Hence \eqref{symmetryPolynomials} holds also for $a = 0$. The solution $u(x;0,0) = x$ has a singular point $x = 0$, where the coefficient at the leading derivative in ODE \eqref{generalODE} vanishes. The properties of the nearby solutions $u(x;0,a \cdot P)$ of \eqref{generalODE} with small $|a|$ in the neighbourhood of this singular point may then be related to the conditions on $P$.

Let us consider solutions $u(x)$ of ODE \eqref{generalODE} near a zero $x = x_0$, i.e., at which $u(x_0) = 0$. We then ask under which conditions on $P$ solution $u$ can be holomorphic in the neighbourhood of $x_0$. Inserting a power series
\[ u(x) = \sum_{k=1}^{\infty} c_k (x-x_0)^k 
\]
into ODE \eqref{generalODE}, we obtain at order 0
\[ -c_1^2 = e^{-x_0}P(0) - 1.
\]
Hence $c_1 = \pm\sqrt{1 - e^{-x_0}\pi_0}$. Condition $\pi_0 = 0$ then ensures that we may choose $c_1 = u'(x_0) = 1$, matching the derivative of $u(x;0,0) = x$.

Let us now assume that $\pi_0 = 0$, $c_1 = 1$, and consider the next orders. At orders 1 and 2 we obtain with $\delta = x - x_0$
\[ (2c_2 + 6c_3\delta)(\delta + c_2\delta^2) - (1 + 2c_2\delta + 3c_3\delta^2)^2 - e^{-x_0}(1 - \delta)(\pi_1(\delta + c_2\delta^2) + \pi_2\delta^2) + 1 = O(\delta^3).
\]
After simplification this gives
\[ - 2c_2^2\delta - 2c_2 - e^{-x_0}(1 - \delta)(\pi_1 + \pi_1c_2\delta + \pi_2\delta) = O(\delta^2).
\]
This yields $c_2 = -\frac12e^{-x_0}\pi_1$ for order 1 and
\[ - 2c_2^2 - e^{-x_0}(\pi_1c_2 + \pi_2 - \pi_1) = 0
\]
for order 2, which is equivalent to the condition $\pi_1 = \pi_2$.

Each subsequent order $k \geq 3$ then gives a condition on the coefficient $c_{k+1}$, while the coefficient $c_3$ remains free. 

We thus obtain exactly the conditions on the coefficients of $P$ which guarantee the existence for the formal series \eqref{fullAsymptotics} for solutions of ODE \eqref{generalODE}.

The next result shows that under these conditions on $P$ there indeed exists a 1-parametric family of holomorphic solutions of ODE \eqref{generalODE} with a zero at $x_0$.

\begin{lemma} \label{lem:solutionWithZero}
    Let $x_0,v_0,\pi_1 \in \mathbb C$ be arbitrary. Let $P(t) = \pi_1t(t+1) + \tilde P(t) \cdot t^3$, where $\tilde P(t)$ is a polynomial. Then there exists a neighbourhood $U \subset \mathbb C$ of $x_0$ such that ODE \eqref{generalODE} has a unique holomorphic solution in $U$ with expansion
    \[ u(x) = (x - x_0) - \frac12e^{-x_0}\pi_1(x - x_0)^2 + \left(\frac16e^{-2x_0}\pi_1^2 + \frac13e^{-x_0}\pi_1 + \frac13v_0\right)(x - x_0)^3 + \dots
    \]
\end{lemma}

\begin{proof}
    Consider the ODE
    \[ \begin{pmatrix}
        u \\ v
    \end{pmatrix}' = \begin{pmatrix}
        1 - e^{-x}\pi_1 u + u^2v \\ e^{-x}\tilde P(u) + e^{-x}\pi_1v - uv^2
    \end{pmatrix}.
    \]
    The solution of this ODE with initial condition $(u(x_0),v(x_0)) = (0,v_0)$ exists, is unique and holomorphic in some neighbourhood $U$ of $x_0$.

    By differentiating $u'$ we obtain 
    \[ u'' = e^{-x}\pi_1 u - e^{-x}\pi_1u' + 2uu'v + u^2v' = e^{-x}\pi_1(u - 1 + e^{-x}\pi_1 u - 2u^2v) + 2uv + u^3v^2 + e^{-x}u^2\tilde P(u).
    \]
    Plugging the expressions for $u',u''$ into \eqref{generalODE}, it is verified that $u(x)$ is indeed a solution of \eqref{generalODE}.

Differentiating $u'' = e^{-x}\pi_1(u - 1 + e^{-x_0}\pi_1 u) + 2uv_0 + O((x-x_0)^2)$ further gives
\[ u''' = e^{-x}\pi_1 + e^{-x}\pi_1(1 + e^{-x_0}\pi_1) + 2v_0 + O(x - x_0).
\]
    Evaluating the derivatives of $u$ at $x = x_0$ yields
    \[ u'(x_0) = 1, \quad u''(x_0) = -e^{-x_0}\pi_1, \quad u'''(x_0) = e^{-2x_0}\pi_1^2 + 2e^{-x_0}\pi_1 + 2v_0,
    \]
    which gives the claimed Taylor expansion.
\end{proof}

\section{Expansions of Painlev\'e transcendents} \label{sec:Painleve}

In \cite[Section 2]{NecessaryConditions} it has been shown that for $P(t) = t^d$, $d = 3,4,6$, ODE \eqref{generalODE} is equivalent to certain degenerate Painlev\'e III equations. This implies that formal expansions can be derived from \eqref{fullAsymptotics} for the corresponding Painlev\'e transcendents. In all cases the coefficients of the first few polynomials $p_k(x;t^d)$ are provided in \cite[Section 6]{NecessaryConditions}.

In \cite[Conjecture C.1]{KitaevVartanianArxiv25} it has been already suggested that such series exist for special families of solutions of other degenerate Painlev\'e III equations.

\subsection{Case $P(t) = t^3$}

The solutions $y(t)$ of the Painlev\'e III type $D_7$ equation 
\begin{equation} \label{D7equation}
ty\ddot y = t\dot y^2 - y\dot y + 4y^3 - 4t,    
\end{equation}
can be obtained from solutions $u(x)$ of ODE \eqref{generalODE} with $P(t) = t^3$ by the transformation $y(t) = t \cdot u(-2\log t)$. Inserting expansion \eqref{fullAsymptotics}, we obtain the asymptotic expansion
\begin{equation*} 
    y(t) = \sum_{k=0}^{\infty} p_k(-2\log\,t + c;t^3) \cdot t^{2k+1},
\end{equation*} 
where $c \in \mathbb C$ is an arbitrary complex parameter. 

\subsection{Case $P(t) = t^4$}

The solutions $y(t)$ of the Painlev\'e III type $D_8$ equation 
\begin{equation*} 
    ty\ddot y = t\dot y^2 - y\dot y + 8y^3 - 8y
\end{equation*}
can be obtained from solutions $u(x)$ of ODE \eqref{generalODE} with $P(t) = t^4$ by the transformation $y(t) = t \cdot u(-2\log t)^2$. Inserting expansion \eqref{fullAsymptotics}, we obtain
\begin{equation*} 
    y(t) = t \cdot \left( \sum_{k=0}^{\infty} p_k(-2\log\,t + c;t^4) \cdot t^{2k} \right)^2 = \sum_{s=0}^{\infty} \left( \sum_{k+l=s} p_k(-2\log\,t + c;t^4)p_l(-2\log\,t + c;t^4) \right) \cdot t^{2s+1},
\end{equation*} 
where $c \in \mathbb C$ is an arbitrary complex parameter. 

\subsection{Case $P(t) = t^6$}

Another family of solutions $y(t)$ of Painlev\'e III equation \eqref{D7equation} can be obtained from solutions $u(x)$ of ODE \eqref{generalODE} with $P(t) = t^6$ by the transformation $y(t) = 8t^{-1} \cdot u(-4\log t+9\log\,2)^{-2}$. Inserting expansion \eqref{fullAsymptotics}, we obtain
\begin{equation*} 
    y(t) = \frac{8}{t} \left( \sum_{k=0}^{\infty} p_k(-4\log t + 9\log\,2 + c;t^6) \cdot \frac{t^{4k}}{2^{9k}} \right)^{-2},
\end{equation*} 
where $c \in \mathbb C$ is an arbitrary complex parameter. 

\section{Conclusion and outlook} \label{sec:outlook}

The results in this work ensure that under the conditions $\pi_0 = 0$, $\pi_1 = \pi_2$ on the coefficients $\pi_j$ of the polynomial $P$, if a solution $u(x)$ with asymptotics \eqref{roughAsymptotics} exists, then it has an asymptotic expansion of the form \eqref{fullAsymptotics}, where the $p_k$ are appropriate polynomials. It does not follow that the asymptotic series actually converges.

However, numerical experiments suggest not only convergence of the series, but also nonnegativity of the coefficients of $p_k$ in case the coefficients $\pi_j$ are nonnegative. In particular, this would imply corresponding series representations for the Painlev\'e transcendents mentioned in Section \ref{sec:Painleve}. 

Convergence of exponential-polynomial series like \eqref{fullAsymptotics} has been studied, e.g., in \cite{Krivosheyeva13,Krivosheevy22}. 

The conditions on the coefficients $\pi_j$ turn out to be also necessary and sufficient for the existence of holomorphic solutions $u(x)$ of ODE \eqref{generalODE} in the neighbourhood of a zero $x_0 \in \mathbb C$ with derivative $u'(x_0) = 1$. This suggests a deeper relation between the asymptotic behaviour for $Re\,x \to +\infty$ and the behaviour near zeros.

\bibliographystyle{plain}
\bibliography{exponential_polynomial,Painleve}

\end{document}